\documentclass{amsart}

\usepackage{amsmath}
\usepackage{amssymb}
\usepackage{amsthm}
\usepackage{mathtools}

\usepackage{tikz}
\usepackage{subcaption}

\newtheorem{theorem}{Theorem}
\newtheorem{lemma}[theorem]{Lemma}
\newtheorem*{theorem*}{Theorem}

\usepackage{myshortcuts}

\begin{document}

\title{A Sharp Fourier Inequality \\and the Epanechnikov Kernel}
\author{Sean Richardson}
\address{Sean Richardson, Department of Mathematics, University of Washington, Seattle, WA}
\email{seanhr@uw.edu}
\pagestyle{plain}

\begin{abstract}
	We consider functions $f: \Z \to \R$ and kernels $u: \{-n, \cdots, n\} \to \R$ normalized by $\sum_{\ell = -n}^{n} u(\ell) = 1$, making the convolution $u \ast f$ a ``smoother'' local average of $f$. We identify which choice of $u$ most effectively smooths the second derivative in the following sense. For each $u$, basic Fourier analysis implies there is a constant $C(u)$ so  $\ltwo{\Delta(u \ast f)} \leq C(u)\ltwo{f}$ for all $f: \Z \to \R$. By compactness, there is some $u$ that minimizes $C(u)$ and in this paper, we find explicit expressions for both this minimal $C(u)$ and the minimizing kernel $u$ for every $n$. The minimizing kernel is remarkably close to the Epanechnikov kernel in Statistics. This solves a problem of Kravitz-Steinerberger and an extremal problem for polynomials is solved as a byproduct.
\end{abstract}

\maketitle

\section{Introduction}
We are interested in the study of functions $f:\mathbb{Z} \rightarrow \mathbb{R}$. These functions appear naturally in many applications (for example as ``time series'') and a natural problem that arises frequently is to take local averages at a fixed, given scale. A popular way is to fix a kernel $u: \{-n, -n+1, \cdots, n-1, n\} \to \R$ and consider the convolution
$$ (u \ast f)(k) = \sum_{\ell = -n}^{n} u(\ell) f(k-\ell).$$
A natural question is now which kernel $u: \{-n, \cdots, n\} \to \R$ one should pick. We will always assume that the kernel is normalized $\sum_{\ell = -n}^{n} u(\ell) = 1$ so that the convolution $u \ast f$ is indeed a local average. There is no ``right choice'' of a kernel $u$ as different choices of weights are optimal in different ways. For example, in image processing theory \cite{lindeberg-dss} one is interested in kernels $u$ so that for any $f$, the convolution $u \ast f$ has fewer local extrema than $f$; this property together with the other ``scale-space axioms'' uniquely characterizes the Gaussian kernel \cite{uniqueness-of-gaussian, koenderink1984structure, scaling-thms}. In kernel density estimation, one is interested in a kernel that minimizes least-squared error, which uniquely characterizes the Epanechnikov kernel \cite{epanechnikov}. In this paper, we build on the work of Kravitz and Steinerberger \cite{smoothest-avg} and take the approach of asking that the convolution $u \ast f$ be as smooth as possible, which we show uniquely characterizes yet another kernel in the second derivative case.

To make this precise, first define the discrete derivative $Df:\mathbb{Z} \rightarrow \mathbb{R}$ by $Df(k) = f(k+1) - f(k)$ and define higher-order derivatives inductively by $D^m f = D (D^{m-1}f)$. Then it follows from basic Fourier analysis (see Section~\ref{sec:avg-to-poly}) that for every kernel $u: \{-n, \cdots, n\} \to \R$ and every $m \in \mathbb{N}$, there exists a constant $C_{m}(u) < \infty$ so that
$$ \forall~f \in \ell^2(\mathbb{Z}), \qquad \ltwo{D^{m}(u \ast f)} \leq C_m(u) \ltwo{f}.$$
We can now ask a natural question.
\begin{quote}
\textbf{Question.} Given a positive integer $m$, how small can $C_{m}(u)$ be and which convolution kernels $u$ attain the optimal constant?
\end{quote}
Such a kernel would then be the ``canonical'' kernel producing the smallest $m$th derivatives and is a natural candidate for use in practice. The problem has been solved by Kravitz and Steinerberger when $m=1$.
\begin{theorem*}[Kravitz-Steinerberger \cite{smoothest-avg}]
For any normalized $u: \{-n, \cdots, n\} \to \R$,
$$C_{1}(u)  \geq \frac{2}{2n+1}$$
with equality if and only if $u(k) = 1/(2n+1)$ is the constant kernel.
\end{theorem*}

That is, averaging by convolving with the characteristic function of an interval best minimizes the first derivative. Kravitz and Steinerberger also studied the $m=2$ case under the assumption $u$ has non-negative Fourier transform.

\begin{theorem*}[Kravitz-Steinerberger \cite{smoothest-avg}]
For any normalized $u: \{-n, \cdots, n\} \to \R$ with nonnegative Fourier transform,
$$C_{2}(u) \geq \f{4}{(n+1)^2}$$
with equality if and only if $u$ is the triangle function $u(k) = (n+1-|k|)/(n+1)^2.$
\end{theorem*}

The main result of this paper resolves the $m=2$ case without additional assumptions on the kernel, providing the optimal constant and optimal kernel for all $n$.

\begin{theorem*}[Main Result]
For any normalized $u: \{-n, \cdots, n\} \to \R,$
$$ C_2(u)  \geq \frac{4}{n+1}\f{\sin(\frac{\pi}{2n+2})}{1+\cos(\frac{\pi}{2n+2})}$$
with equality if and only if $u$ is given by the kernel in (\ref{eq:optimal-u}). 
\nonumber
\end{theorem*}

A quick computation shows the optimal kernels $u_n: \{-n, \cdots, n\} \to \R$ satisfy, as $n \to \infty$, the following asymptotic equivalence; notice the asymptotic improvement from $4$ to $\pi$ when removing the nonnegative Fourier transform restriction:
\begin{equation}
	C_{2}(u_n) \sim \f{\pi}{(n+1)^2}.
	\label{eq:optimal-asymptotic}
\end{equation}

The optimal kernel $u_n$ for each $n$ is given as an integral expression by (\ref{eq:optimal-u}) in the following section, and Figure~\ref{fig:optimal-kernel} pictures the optimal kernels $u_{10}$ and $u_{1000}$. Figure~\ref{fig:chelan-smoothing} depicts a time series function $f: \Z \to \R$ encoding the water level of Lake Chelan over a two week period as well as the smoothed data $u_{10} \ast f$. Notice this smoothing reduces noise and clarifies long-term trends. 

\begin{figure}[!ht]
    \centering
    \begin{subfigure}{0.41\textwidth}
    	\includegraphics[width=\textwidth]{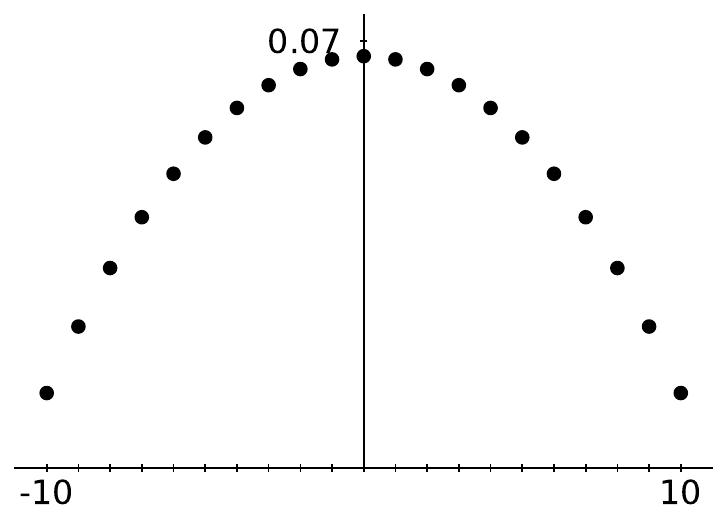}
    	\caption{$u_{10}(k)$}
    \end{subfigure}
    \qquad\tikz[baseline=-\baselineskip]\draw[white,ultra thick,->] (0,0) -- ++ (0.5,0);\qquad
    \begin{subfigure}{0.41\textwidth}
        \includegraphics[width=\textwidth]{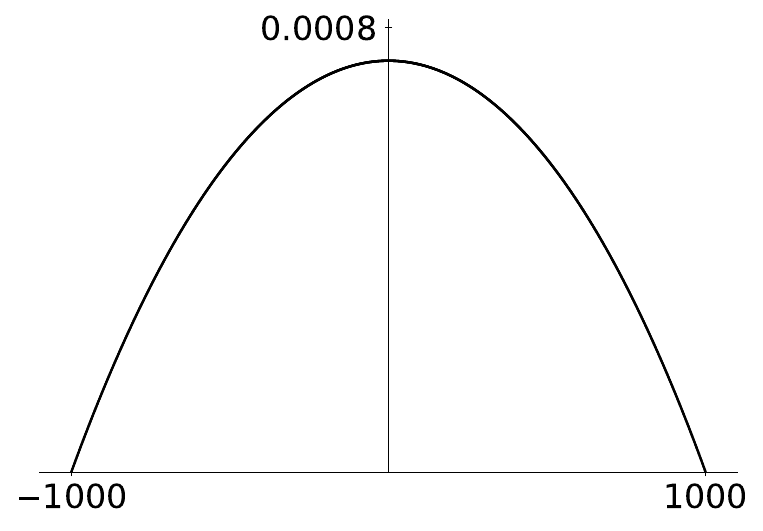}
         \caption{$u_{1000}(k)$}
	\end{subfigure}
	\caption{The optimal kernel $u_n(k)$}
	\label{fig:optimal-kernel}
\end{figure}
\begin{figure}[!ht]
    \centering
    \begin{subfigure}{0.41\textwidth}
    	\includegraphics[width=\textwidth]{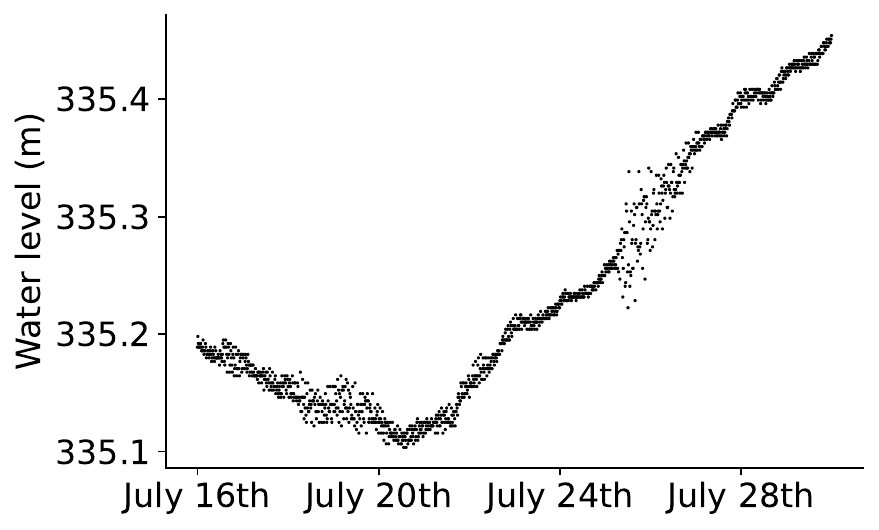}
    	\caption{Original data $f$}
    \end{subfigure}
\qquad\tikz[baseline=-\baselineskip]\draw[ultra thick,->] (0,0) -- ++ (0.5,0);\qquad
    \begin{subfigure}{0.41\textwidth}
        \includegraphics[width=\textwidth]{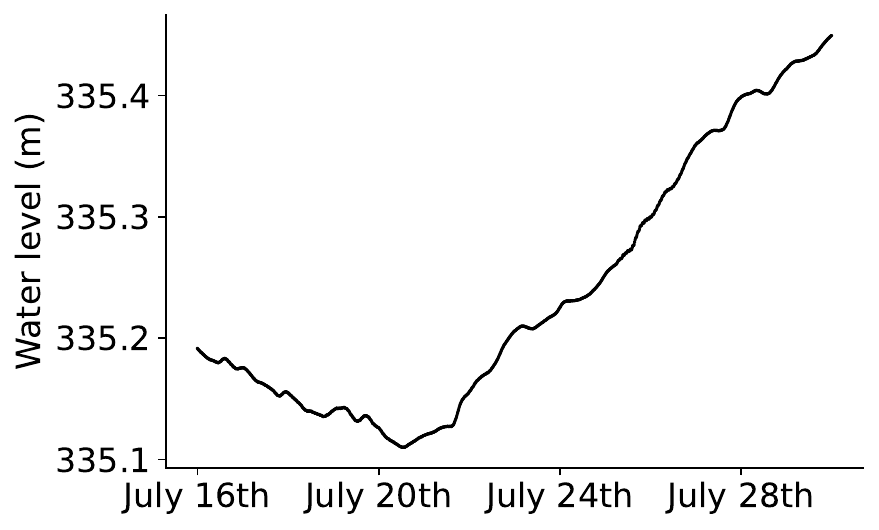}
         \caption{Smoothed data $u_{10} \ast f$}
	\end{subfigure}
	\caption{Water level of Lake Chelan over time \cite{chelan-data}}
	\label{fig:chelan-smoothing}
\end{figure}

As seen in Figure~\ref{fig:optimal-kernel}, the optimal kernel resembles a parabola, but it turns out the points do not quite lie on any parabola. However, choosing weights by sampling from a parabola results in the discrete Epanechnikov kernel $E_n$ for each $n$, which is a simple and effective approximation of the optimal kernel. Indeed, we show the parabolic Epanechnikov kernel has constant $C_2(E_n)$ within $2\%$ of the optimal constant $C_2(u_n)$ for large $n$, providing a new reason to use the Epanechnikov kernel.

\section{Results}

\subsection{A Sharp Fourier Inequality}

\label{sec:smooth-averaging}

We can rewrite the main result discussed in the Introduction as the following Fourier inequality. As usual, we only consider kernels $u: \{-n, -n+1, \cdots, n-1, n\} \to \R$ normalized so that $\sum_{\ell=-n}^n u(\ell) = 1$.
\begin{theorem}[Main Result, restated]
For any normalized $u:\left\{-n, \dots, n \right\} \rightarrow \mathbb{R}$,
\begin{equation}
	\sup_{0 \neq f \in \ell^2(\Z)} 
	\f{\ltwo{\Delta(u \ast f)}}{\ltwo{f}} \geq  \frac{4}{n+1} \cdot \f{\sin\l(\frac{\pi}{2n+2}\r)}{1+\cos\l(\frac{\pi}{2n+2}\r)}.
	\label{eq:main-inequality}
\end{equation}
	With equality if and only if $u(k)$ is as in (\ref{eq:optimal-u}).
	\label{thm:kernel-inequality}
\end{theorem}

Note the discrete Laplacian $(\Delta f)(k) \coloneqq f(k+2) - 2f(k+1) + f(k)$ is precisely the discrete second derivative $(D^2f)(k)$ defined in the Introduction. The optimal $u_n: \{-n, \cdots, n\} \to \R$ that yields equality in Theorem~\ref{thm:kernel-inequality} for any $n$ can be written explicitly by defining $u_n$ to be symmetric $u_n(k) = u_n(-k)$, then for $k \geq 0$ setting
\begin{equation}
		u_n(k) = \f{1}{\pi} \int_{-1}^1 S_n(x) T_k(x) \f{dx}{\sqrt{1-x^2}}
		\label{eq:optimal-u}
\end{equation}
where $T_k(x)$ is the $k$th Chebyshev polynomial (as defined in Section \ref{sec:polynomial}), and
\begin{equation}
S_{n-1}(x) = \f{1}{x-1} \cdot \f{2\sin\l(\f{\pi}{2n}\r)}{n\l(1+\cos\l(\f{\pi}{2n}\r)\r)} \cdot T_{n}\l(\f{1+\cos\l(\f{\pi}{2n}\r)}{2}(x+1)-1\r).
\label{eq:Sn}
\end{equation}

This result extends the work of Kravitz and Steinerberger \cite{smoothest-avg} and adds to the recent research activity on sharp Fourier inequalities. For example, there is current research on sharp Fourier restriction and extension inequalities \cite{carneiro2019extremizers, carneiro2017sharp, oliveira2022band, stovalluniform, stovall2020extremizability}, on sharp Strichartz inequalities \cite{foschi2007maximizers, hundertmark2006sharp}, on sharp Hausdorff-Young inequalities \cite{kovavc2019sharp, kovavc2022asymptotically}, and other Fourier inequalities \cite{inequalities-in-fa, beckner1995pitt, bandwith-vs-time, hermite-polynomials}

\subsection{The Epanechnikov Kernel}
The optimal kernel $u_n$ has a complicated expression as given in (\ref{eq:optimal-u}). However, as seen in Figure~\ref{fig:optimal-kernel}, this optimal $u_n$ resembles a parabola, and conversely we find that choosing weights by sampling from a parabola does extraordinarily well in smoothing the Laplacian of a given function. This choice of weights yields the discrete normalized \textit{Epanechnikov kernel} $E_n: \{-n, \cdots, n\} \to \R$ defined by
\begin{equation}
	E_n(k) = \f{3}{n (4 n^2 - 1)}\l(n^2-k^2\r).
	\label{eq:epanechnikov}
\end{equation}

The Epanechnikov kernel is widely used \cite{hall2004bump, samiuddin1990nonparametric, yamasaki2019kernel} in both theory and applications. This popularity stems from its computational efficiency and from Epanechnikov's proof \cite{epanechnikov} that it is the optimal kernel to use in kernel density estimation (KDE) in terms of minimizing expected mean integrated square error. The following theorem reveals the Epanechnikov kernel is less than 2\% worse than optimal in smoothing the Laplacian, providing another reason to use the Epanechnikov kernel in practice.

\begin{theorem}
Let $\mu$ be as in (\ref{eq:mu-def}). Then as $n \to \infty$ we get the asymptotic equivalence
\begin{equation*}
	\sup_{0 \neq f \in \ell^2(\Z)} \f{\ltwo{\Delta(E_n \ast f)}}{\ltwo{f}}
	\sim \f{3\mu}{\pi} \cdot \f{\pi}{n^2}.
\end{equation*}
\label{thm:epanechnikov}
\end{theorem}
The constant $3\mu/\pi$ in the above theorem is a universal constant defined by
\begin{equation}
	 \f{3\mu}{\pi} \coloneqq \f{3}{\pi} \max_{\alpha \in [0,16]}\l|\f{\sin\alpha}{\alpha} - \cos{\alpha}\r| \approx  1.015.
	 \label{eq:mu-def}
\end{equation}

Comparing Theorem~\ref{thm:kernel-inequality}, whose asymptotics are given by (\ref{eq:optimal-asymptotic}), and the statement of Theorem~\ref{thm:epanechnikov} reveals the asymptotics of $C_2(u_n)$ and $C_2(E_n)$ differ only by a factor of $3 \mu / \pi \approx 1.015$ as $n \to \infty$. Hence the Epanechnikov kernel performs less than $2\%$ worse than optimal asymptotically. 

\subsection{A Sharp Polynomial Inequality}
After taking a Fourier transform (see Section~\ref{sec:avg-to-poly}), Theorem~\ref{thm:kernel-inequality} reduces to the following claim about polynomials, which is interesting in its own right.
\begin{theorem}
	Let $p(x)$ be a polynomial of degree at most $n-1$ with $p(1) = 1$. Then
$$\max_{x \in [-1,1]}|(1-x)p(x)| \geq \f{2\sin\l(\f{\pi}{2n}\r)}{n\l(1+\cos\l(\f{\pi}{2n}\r)\r)}$$
with equality if and only if $p(x) = S_{n-1}(x)$ where $S_{n-1}(x)$ is given in (\ref{eq:Sn}).
\label{thm:min-poly}
\end{theorem}

\begin{figure}[!ht]
    \centering
    \begin{subfigure}{0.41\textwidth}
    	\includegraphics[width=\textwidth]{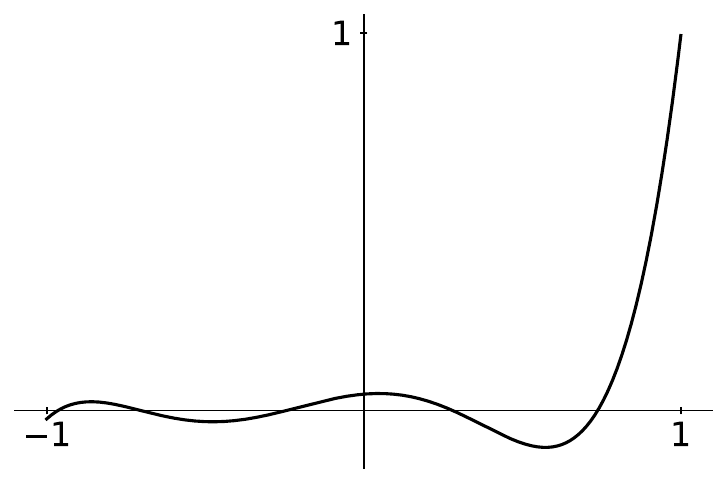}
    	\caption*{Optimal polynomial $S_5(x)$}
    \end{subfigure}
    \qquad\tikz[baseline=-\baselineskip]\draw[white,ultra thick,->] (0,0) -- ++ (0.5,0);\qquad
    \begin{subfigure}{0.41\textwidth}
        \includegraphics[width=\textwidth]{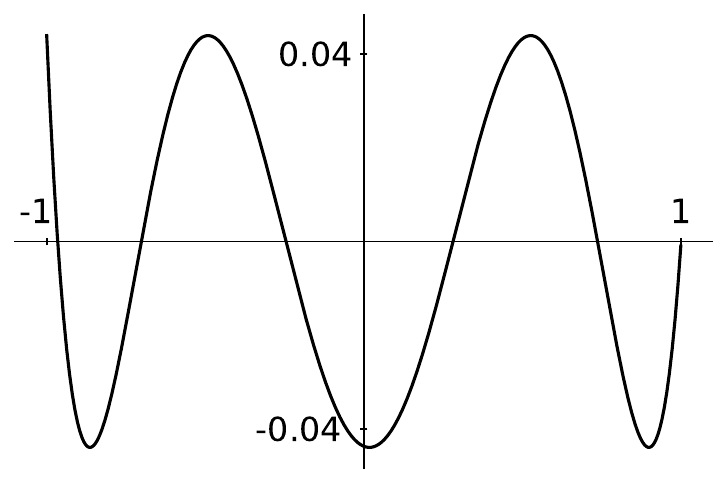}
         \caption*{Polynomial $(1-x)S_5(x)$}
	\end{subfigure}
	\caption{Optimal Polynomial and Product}
\end{figure}

\section{Proofs}
\subsection{A Sharp Polynomial Inequality}
\label{sec:polynomial}
This section proves Theorem~\ref{thm:min-poly}, which is the key to the proof of Theorem~\ref{thm:kernel-inequality} and the formula for the optimal kernel given in (\ref{eq:optimal-u}). This proof makes heavy use Chebyshev polynomials, so we first recall some facts about Chebyshev polynomials for the convenience of the reader.

\textit{Chebyshev polynomials} are a family $\{T_n(x)\}$ of polynomials defined by declaring $T_0(x) = 1$ and $T_1(x) = x$, then defining the rest with the recurrence 
\begin{equation*}
T_{n+1}(x) = 2xT_{n}(x)-T_{n-1}(x).
\end{equation*}
A quick induction argument reveals Chebyshev polynomials also satisfy
\begin{equation}
T_{n}(\cos\theta) = \cos(n\theta).
\label{eq:cosines}
\end{equation}
The above relation is responsible for many nice properties of Chebyshev polynomials and so we typically consider Chebyshev polynomials $T_n(x)$ over the domain $[-1,1]$ where this formula can apply. This relation also reveals the zeroes of the Chebyshev polynomial $T_n(x)$ are located at $\cos\l(\f{\pi}{n}(k+\f{1}{2})\r)$ for $k \in \Z$. 
By taking the derivative of both sides of (\ref{eq:cosines}), we find that the derivative $T_n'(x)$ of Chebyshev polynomials over $[-1,1]$ can be written $T_n'(x) = nU_{n-1}(x)$ for some polynomial $U_{n-1}(x)$ that satisfies $U_{n-1}(\cos(\theta))\sin\theta = \sin(n\theta)$. These polynomials $\{U_{n}(x)\}$ are called \textit{Chebyshev polynomials of the second kind}. Finally, due to (\ref{eq:cosines}), each Chebyshev polynomial $T_n(x)$ satisfies the \emph{equioscillation property}, meaning there exists $n+1$ extrema $1 = x_{0} > x_1 > \cdots > x_{n} = -1$ so that $T_n(x_i) = (-1)^{i}$.

We are now equipped to prove Theorem~\ref{thm:min-poly}. The equation defining the optimal polynomial $S_{n-1}(x)$ is long, but the idea behind its construction is a simple modification of Chebyshev polynomials. To construct a degree $n$ function $(1-x)p(x)$ that stays minimal over $[-1,1]$, we slightly stretch the Chebyshev polynomial $T_n(x)$ to the function $q_n(x)$ so that $q_n(1) = 0$ and $q_n'(1) = -1$. Then defining $S_n(x) = q_n(x)/(1-x)$ provides the minimal degree $n-1$ polynomial that satisfies the necessary conditions.

\begin{proof}[Proof of Theorem~\ref{thm:min-poly}]
	We start by verifying the claimed optimal polynomial $S_{n-1}(x)$ fulfills the restrictions and satisfies the claimed inequality for any positive integer $n$. We denote $(1-x)S_{n-1}(x)$ by $q_n(x)$ and observe $q_n(x) = -\alpha T_n(L(x))$ is a scaled and stretched Chebyshev polynomial for carefully chosen constant $\alpha$ and linear change of variables $L(x)$:
	\begin{align*}
		\alpha = \f{2\sin\l(\f{\pi}{2n}\r)}{n\l(1+\cos\l(\f{\pi}{2n}\r)\r)} \quad \text{and} \quad
		L(x) = \f{1+\cos\l(\f{\pi}{2n}\r)}{2}(x+1)-1.
	\end{align*}
	Writing $q_n(x) = -\alpha T_n(L(x))$ immediately reveals $q_n(x)$ is a polynomial of order $n$. The change of variables $L(x)$ is designed so we have
	$$q_n(1) = -\alpha T_n(L(1)) = -\alpha T_n\l(\cos\l(\f{\pi}{2n}\r)\r) = 0,$$
	using that $\cos\l(\f{\pi}{2n}\r)$ is a zero of the Chebyshev polynomial $T_{n}(x)$. Thus we have $S_{n-1}(x) = q_n(x)/(1-x)$ is a well-defined polynomial of order $n-1$ as claimed. Next we show $S_{n-1}(1) = 1$. First observe $S_{n-1}(1) = -q_n'(1)$ and now compute
	\begin{align*}
		S_{n-1}(1) 
		&= -q_n'(1)
		= -\l.\der{}{x}\r|_{x=1}(-\alpha T_n(L(x)))\\
		&= \alpha L'(1) T_n'(L(1))
		= \alpha \f{n}{2} \l(1+\cos\l(\f{\pi}{2n}\r)\r) U_{n-1}\l(\cos\l(\f{\pi}{2n}\r)\r)
	\end{align*}
	where we used $T_n'(x) = nU_{n-1}(x)$ for $U_{n-1}(x)$ the Chebyshev polynomial of the second kind with order $n-1$. Recall $U_{n-1}(\cos\theta) = \sin(n\theta)/\sin(\theta)$ and therefore, continuing our computation, we see the constant $\alpha$ is chosen so that we get
	\begin{align*}
		S_{n-1}(1)
		= \alpha \cdot \f{n}{2} \cdot \f{1+\cos\l(\f{\pi}{2n}\r)}{\sin\l(\f{\pi}{2n}\r)}
		= \alpha \cdot \alpha^{-1} 
		= 1.
	\end{align*}
	Therefore $S_{n-1}(x)$ fulfills the necessary conditions. To verify $S_{n-1}(x)$ satisfies the equality, use $L([-1,1]) \subset [-1,1]$ and that $|T_n(y)| \leq 1$ on $[-1,1]$ to see
\begin{align*}
	\max_{x \in [-1,1]}|(1-x)S_{n-1}(x)| 
	&= \max_{x \in [-1,1]}|q_n(x)|\\
	&= \max_{x \in [-1,1]}|\alpha T_n(L(x))|
	\leq \alpha \max_{y \in [-1,1]}|T_n(y)|
	\leq \alpha.
\end{align*}
	This is indeed an equality because
\begin{align*}
	\max_{x \in [-1,1]}|(1-x)S_{n-1}(x)| \geq  |\alpha T_n(L(-1))| = \alpha\cdot|T_n(-1)| = \alpha.
\end{align*}

	We now show $S_{n-1}(x)$ is the unique polynomial of degree at most $n-1$ that achieves this equality by deriving an equioscillation property for $q_n(x)$ and modifying the standard argument for the minimizing property of Chebyshev polynomials. Recall the Chebyshev polynomial $T_n(x)$ has $n+1$ extrema $1 = x_{0} > x_1 > \cdots > x_{n} = -1$ so that the equioscillation property $T_n(x_i) = (-1)^{i}$ is satisfied. To see $q_n(x)$ has a similar equioscillation property, define inputs $y_i = L^{-1}(x_i)$ for $0 \leq i \leq n$, and observe $y_0 \geq \cdots \geq y_{n}$ by $L^{-1}(x)$ strictly increasing. Next, note the second extrema of Chebyshev polynomials is given by $x_1 = \cos\l(\f{\pi}{n}\r)$ and so $x_1 < \cos\l(\f{\pi}{2n}\r)$; thus by $L^{-1}(x)$ strictly increasing, $y_1 = L^{-1}(x_1) < L^{-1}(\cos\l(\f{\pi}{2n}\r)) = 1$. Now compute $L(-1) = -1$, which implies $L^{-1}(-1) = -1$. Therefore we find the inputs $y_i$ satisfy $1 > y_1 \geq \cdots \geq y_n = -1$, and $q_n(y_i) = -\alpha T_n(L(y_i)) = -\alpha T_n(x_i) = -\alpha (-1)^{i}$. That is, $q_n(x)$ has $n$ maxima satisfying the equioscillation property in $[-1,1]$. Next suppose $p(x)$ is any polynomial of degree $n-1$ so that $p(1) = 0$ and
	$$\max_{x \in [-1,1]}|(1-x)p(x)| \leq \alpha.$$
	We argue $p(x)=S_{n-1}(x)$ by using our equioscillation property to show the polynomials must intersect sufficiently many times. Formally, we count the zeros of the polynomial $z(x) = q_n(x)-(1-x)p(x)$. By our equioscillation property $q_n(y_i) = -\alpha(-1)^{i}$ and the assumption $|(1-x)p(x)| \leq \alpha$ over $[-1,1]$, we require $z(y_i) \geq 0$ for $i$ odd and $z(y_i) \leq 0$ for $i$ even. Therefore $z(x)$ must have a zero in each interval $[y_i, y_{i+1}]$ by the intermediate value theorem. That is, the $n-1$ intervals $[y_1, y_2], \cdots, [y_{n-1},y_{n}]$ all contain a zero. Furthermore if any two intervals $[y_{i-1},y_i]$ and $[y_i,y_{i+1}]$ share a zero at $y_i$, then we can show $z(x)$ will have a zero of multiplicity at least two at $y_i$ and therefore $z(x)$ still has at least $n-1$ zeros counted with multiplicity on $[y_0,y_{n-1}]$ as follows. By $q_n(y_i) = \pm \alpha$ at $y_i$, we know $q_n(x)$ has a minimum or maximum at $y_i$, implying $q_n'(y_i) = 0$. Similarly, if $z(y_i) = 0$, then $(1-y_i)p(y_i) = \pm \alpha$ and so we also find that $(1-x)p(x)$ has a minimum or maximum at $y_i$, so $\derev{}{x}{y_i}(1-x)p(x) = 0$. Therefore $z'(y_i) = q_n'(y_i) - \derev{}{x}{y_i}(1-x)p(x) = 0$ and so $z(x)$ indeed has a zero of multiplicity at least two at $y_i$ and so $z$ has $n-1$ zeros on $[-1,1)$.
	
	\begin{figure}[!ht]
    \centering
    \begin{subfigure}{0.41\textwidth}
    	\includegraphics[width=\textwidth]{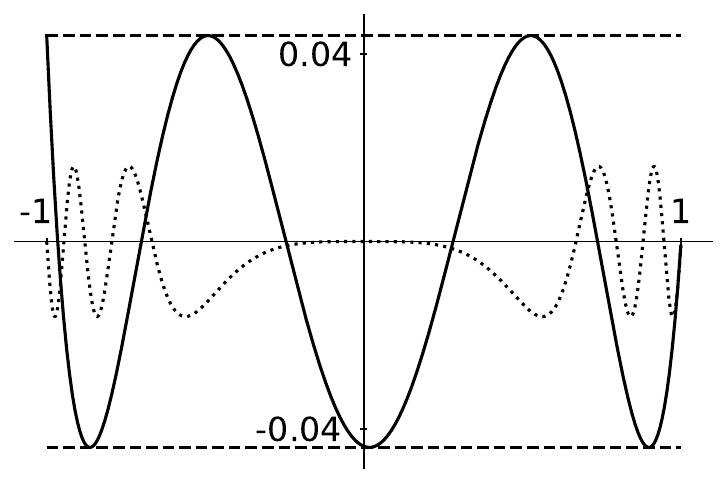}
    	\caption*{$q_n(x)$ and $(1-x)p(x)$ must intersect}
    \end{subfigure}
    \qquad\tikz[baseline=-\baselineskip]\draw[white,ultra thick,->] (0,0) -- ++ (0.5,0);\qquad
    \begin{subfigure}{0.41\textwidth}
        \includegraphics[width=\textwidth]{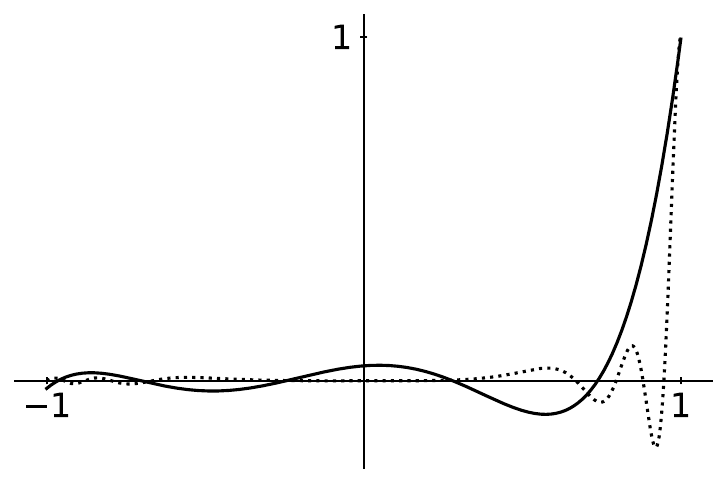}
         \caption*{$S_{n-1}(x)$ and $p(x)$ must intersect}
	\end{subfigure}
\end{figure}
	
	Now note $z(1) = q_n(1)-(1-1)p(1) = 0-0 = 0$ and so $z(x)/(1-x)$ is a polynomial of degree $n-1$ with the same roots on $[-1,1)$. Additionally observe
	$$\f{z(x)}{1-x} = \f{q_n(x)-(1-x)p(x)}{1-x} = S_{n-1}(x)-p(x)$$
	and therefore, evaluating $z(x)/(1-x)$ at $1$ reduces to $S(1)-p(1) = 1-1 = 0$ and so $z(x)/(1-x)$ has $n$ roots. However because $z(x)/(1-x)$ is of degree $n-1$ this implies $z(x)/(1-x)$ is the zero polynomial. That is, $p(x) = S_{n-1}(x)$.
\end{proof}
\subsection{Reduction to symmetric kernels}
Our Fourier analysis argument for Theorem~\ref{thm:kernel-inequality} given in Section~\ref{sec:avg-to-poly} only holds for \textit{symmetric kernels}, satisfying $u(k) = u(-k)$. Luckily, the following lemma demonstrates that it is sufficient to only prove Theorem~\ref{thm:kernel-inequality} for the class of symmetric normalized kernels.

\begin{lemma}
	Suppose for all symmetric and normalized kernels $u: \{-n, \cdots, n\} \to \R$,
	\begin{equation}
		\sup_{f \neq 0}\f{\ltwo{\Delta(u \ast f)}}{\ltwo{f}} \geq \beta_n
		\label{eq:sym-ineq}
	\end{equation}
	for some $\beta_n > 0$. Then {\normalfont(\ref{eq:sym-ineq})} holds for all normalized kernels $u: \{-n, \cdots, n\} \to \R$.
\label{lemma:symmetric-kernels}
\end{lemma}
\begin{proof}
	Let $u: \{-n, \cdots, n\} \to \R$ be any kernel normalized by $\sum_{k=-n}^n u(k) = 1$. For any function $g(k)$, define its reflection $g^-(k) = g(-k)$. Next consider the symmetrization kernel $\til{u}(k) = \f{1}{2}(u(k) + u^-(k))$, which will also satisfy $\sum_{k=-n}^n \til{u}(k) = 1$. Now take any $f \in \ell^2(\Z)$ and compute
	\begin{align*}
		\f{\ltwo{\Delta(\til{u} \ast f)}}{\ltwo{f}}
		&= \f{\ltwo{\Delta(\f{1}{2}(u + u^-) \ast f)}}{\ltwo{f}}\\
		&\leq \f{1}{2}\l(\f{\ltwo{\Delta(u \ast f)}}{\ltwo{f}} + \f{\ltwo{\Delta(u^- \ast f)}}{\ltwo{f}}\r)\\
		&= \f{1}{2}\l(\f{\ltwo{\Delta(u \ast f)}}{\ltwo{f}} + \f{\ltwo{\Delta(u \ast f^-)}}{\ltwo{f^-}}\r).
	\end{align*}
	Therefore we find that for all $f \in \ell^2(\Z)$, either
	\begin{equation*}
		\f{\ltwo{\Delta(u \ast f)}}{\ltwo{f}} \geq \f{\ltwo{\Delta(\til{u} \ast f)}}{\ltwo{f}}
		\quad \text{or} \quad
		\f{\ltwo{\Delta(u \ast f^{-})}}{\ltwo{f^{-}}} \geq \f{\ltwo{\Delta(\til{u} \ast f)}}{\ltwo{f}}.
	\end{equation*}
	Hence we can conclude
	\begin{align*}
		\sup_{f \neq 0}\f{\ltwo{\Delta(u \ast f)}}{\ltwo{f}}
		\geq \sup_{f \neq 0}\f{\ltwo{\Delta(\til{u} \ast f)}}{\ltwo{f}} \geq \beta_n.
	\end{align*}
\end{proof}

\subsection{From discrete kernels to polynomial extremizers}
\label{sec:avg-to-poly}
\begin{lemma}[Kravitz and Steinerberger \cite{smoothest-avg}]
	Given a symmetric and normalized kernel $u: \{-n, \cdots, n\} \to \R$, define the polynomial
	\begin{align*}
		p_u(x) = u(0) + \sum_{k=1}^n 2u(k)T_k(x)
	\end{align*}
	where $T_k(x)$ is the $k$th Chebyshev polynomial. Then,
	\begin{align*}
	\sup_{0 \neq f \in \ell^2(\Z)} \f{\ltwo{\Delta(u \ast f)}}{\ltwo{f}}
	= 2 \max_{-1 \leq x \leq 1}|(1-x)p_u(x)|.
	\end{align*}
	\label{prop:avg-to-poly}
\end{lemma}

For the proof of Lemma~\ref{prop:avg-to-poly} we follow the argument given by Kravitz and Steinerberger \cite{smoothest-avg}, which uses the Fourier transform and Plancherel's theorem. Before the proof, first recall that a function $f: \Z \to \R$ on the integers has a continuous Fourier transform $\hat{f}: \T \to \C$ defined on the 1-torus given by
 	\begin{align*}
		\hat{f}(\xi) = \sum_{k \in \Z}f(k) e^{-i\xi k}.
	\end{align*}
This is called the discrete-time Fourier transform, which we will also denote by $\mc{F}(f)(\xi) = \hat{f}(\xi)$. This Fourier transform takes convolution to multiplication by
	\begin{align*}
		\hat{f \ast g} = \hat{f} \cdot \hat{g}.
	\end{align*}	
Another useful property is the Plancherel identity, which relates the inner product of functions with that of their Fourier transform by
	\begin{align*}
		\sum_{k \in \Z} f(k) \ol{g(k)} 
		= \f{1}{2\pi}\int_{\T}\hat{f}(\xi)\ol{\hat{g}(\xi)}d\xi.
	\end{align*}
Note we can express the Fourier transform of a shifted function $f(\cdot-m)$ by
	\begin{align*}
	\hat{f(\cdot-m)}(\xi)
	&= \sum_{k \in \Z}f(k-m)e^{-i\xi k} 
	= \sum_{k \in \Z}f(k)e^{-i\xi(k+m)}\\
	&= e^{-i\xi m} \sum_{k \in \Z}f(k)e^{-i\xi k}
	= e^{-i\xi m}\hat{f}(\xi).
	\end{align*}
Using the above, we can compute the Fourier transform of the discrete Laplacian.
	\begin{align*}
		\mc{F}(\Delta f)(\xi)
		&= \mc{F}(f(k+2)-2f(k+1)+f(k))(\xi)\\
		&= (e^{-2 i \xi}\hat{f}(\xi) - 2e^{-i\xi}\hat{f}(\xi) + \hat{f}(\xi)
		= (e^{-i\xi}-1)^2\hat{f}(\xi).
	\end{align*}

We are equipped to prove Lemma~\ref{prop:avg-to-poly}, but first we follow up with a claim made in the Introduction. Indeed, observe that by the first computation in the following proof, which does not yet use the symmetry of $u$, we can conclude quickly that $\ltwo{\Delta(u \ast f)} \leq C_2(u) \ltwo{f}$ for \emph{some} constant $C_2(u)$ depending continuously on $u$. Furthermore, because the space of normalized kernels ${u: \{-n, \cdots, n\} \to \R}$ is compact, we conclude that there exists some optimal kernel $u_n$ that minimizes $C_2(u)$. This argument can be easily generalized to higher derivatives.

\begin{proof}[Proof of Lemma~\ref{prop:avg-to-poly}]
	Let $u$ be a fixed symmetric, normalized kernel and let $f: \Z \to \R$ be any function in $\ell^2(\Z)$. Plancherel's identity allows us to equate $\ltwo{\Delta (u \ast f)}$ to an easier expression in terms of the Fourier transform by computing
	\begin{align}
		\sum_{k \in \Z}|(\Delta(u \ast f))(k)|^2 \nonumber
		&= \f{1}{2\pi} \int_{\T}|e^{-i\xi}-1|^4|\hat{u}(\xi)|^2 |\hat{f}(\xi)|^2 d\xi \nonumber\\
		&\leq \linfty{(e^{-i\xi}-1)^4\hat{u}(\xi)^2} \cdot \f{1}{2\pi}\int_{\T}|\hat{f}(\xi)|^2d\xi \label{eq:ineq}\\
		&= \linfty{(e^{-i\xi}-1)^4\hat{u}(\xi)^2} \cdot \sum_{k \in \Z}|f(k)|^2. \nonumber
	\end{align}
	After taking a square root we get 
	\begin{align*}
		\ltwo{\Delta(u \ast f)}
		\leq \linfty{(e^{-i\xi}-1)^2\hat{u}(\xi)} \cdot \ltwo{f}
	\end{align*}
	Note that the only inequality in the derivation of the above is in (\ref{eq:ineq}). Furthermore, by choosing $f(k)$ so that $\hat{f}(\xi)$ has $L^2$ mass concentrated at the $\xi$  in which the function $|e^{-i\xi}-1|^4|\hat{u}(\xi)|^2$ achieves it's maximum, we can make the above inequality arbitrary close to an equality.
	Thus we have shown our expression measuring the smoothing of the second derivative is equivalent to the following simpler expression:
	\begin{align*}
		\sup_{0 \neq f \in \ell^2(\Z)}\f{\ltwo{\Delta(u \ast f)}}{\ltwo{f}} = \linfty{(e^{-i\xi}-1)^2\hat{u}(\xi)}.
	\end{align*}
	Denote this simpler expression by
	\begin{align*}
		L(u) \coloneqq \|(e^{i\xi}-1)^2\hat{u}(\xi)\|_{L^{\infty}(\T)}
		= \max_{\xi \in [0,2\pi)} |e^{i\xi}-1|^2|\hat{u}(\xi)|.
	\end{align*}
	Because $u(k)$ is symmetric, real-valued, and only supported on $\{-n, \cdots, n\}$, we can rewrite $\hat{u}(\xi)$ by
	\begin{align*}
		\hat{u}(\xi) 
		= \sum_{k \in \Z} u(k) e^{-i\xi k}
		= u(0) + 2\sum_{k=1}^n u(k)\cos(\xi k).
	\end{align*}
	Therefore we find
	\begin{align*}
		L(u) = \max_{\xi \in [0,2\pi)}|e^{i\xi}-1|^2|\hat{u}(\xi)|
		= 2\max_{\xi \in [0,2\pi)}\l|(1-\cos(\xi))\l(u(0) + 2\sum_{k=1}^n u(k)\cos(\xi k)\r)\r|.
	\end{align*}
	The substitution $x = \cos(\xi)$ gives rise to
	\begin{align*}
		L(u) = 2\max_{-1 \leq x \leq 1} \l|(1-x)\l(u(0) + 2\sum_{k=1}^nu(k)T_k(x)\r)\r|
	\end{align*}
	where $T_k(x)$ denotes the Chebyshev polynomial of degree $k$. Therefore, defining the polynomial $p_u(x)$ by
	\begin{align*}
		p_u(x) = u(0) + \sum_{k=1}^n 2u(k)T_k(x),
	\end{align*}
	we have our desired equality
	\begin{align*}
		\sup_{0 \neq f \in \ell^2(\Z)}\f{\ltwo{\Delta(u \ast f)}}{\ltwo{f}}
		= L(u)
		=  2\max_{-1 \leq x \leq 1} |(1-x) p_u(x)|.
	\end{align*}
\end{proof}

\subsection{The Sharp Fourier Inequality and Optimal Kernel}
\begin{proof}[Proof of Theorem~\ref{thm:kernel-inequality}]
	For any symmetric and normalized $u: \{-n, \cdots, n\} \to \R$, define the degree $n$ polynomial
	\begin{align*}
		p_u(x) = u(0) + \sum_{k=1}^n 2u(k)T_k(x)
	\end{align*}
	as given in Lemma~\ref{prop:avg-to-poly}. Then note $p_u(1) = u(0) + \sum_{k=1}^n 2u(k)T_k(1) = 1$ by the normalization of $u$. Therefore combining Lemma~\ref{prop:avg-to-poly} and Theorem~\ref{thm:min-poly} yields
	\begin{align*}
		\sup_{0 \neq f \in \ell^2(\Z)} \f{\ltwo{\Delta(u \ast f)}}{\ltwo{f}}
		= 2 \max_{-1 \leq x \leq 1}|(1-x)p_u(x)|
		\geq \f{4\sin\l(\f{\pi}{2n+2}\r)}{(n+1)\l(1+\cos\l(\f{\pi}{2n+2}\r)\r)}.
	\end{align*}
	Because the above inequality holds for symmetric normalized kernels, Lemma~\ref{lemma:symmetric-kernels} implies this in fact holds for all normalized kernels.
	To see this inequality is sharp, let $S_n(x)$ be the optimal degree $n$ polynomial as given in Theorem~\ref{thm:min-poly}. Because the Chebyshev polynomials $\{T_0(x), \dots, T_n(x)\}$ form a basis for the space of all degree $n$ polynomials, there exists unique coefficients $\alpha_k$ so that
	\begin{align*}
		S_n(x) = \alpha_0T_0(x) + \sum_{k=1}^n 2\alpha_kT_k(x).
	\end{align*}
	These coefficients define a corresponding kernel $u_n: \{-n, \cdots, n\} \to \R$ by setting $u_n(k) = u_n(-k) = \alpha_k$ for $k \geq 0$. Using $T_k(1) = 1$, we find this kernel is properly normalized by computing
	\begin{align*}
		\sum_{k=-n}^n u_n(k) = \alpha_0 + 2\sum_{k=1}^n \alpha_k  = \alpha_0T_0(1) + \sum_{k=1}^n 2\alpha_kT_k(1) = S(1) = 1.
	\end{align*}
	Noting $T_0(x) \equiv 1$, we can rewrite write $S_n(x)$ as
	\begin{align*}
		S_{n}(x) = u_n(0) + \sum_{k=1}^n 2u_n(k)T_k(x).
	\end{align*}
	Using the equality condition in Theorem~\ref{thm:min-poly}, we find $u_n(k)$ indeed satisfies
	\begin{align*}
		\sup_{0 \neq f \in \ell^2(\Z)} \f{\ltwo{\Delta(u_n \ast f)}}{\ltwo{f}}
		= 2 \max_{-1 \leq x \leq 1}|(1-x)S_n(x)|
		= \f{4\sin\l(\f{\pi}{2n+2}\r)}{(n+1)\l(1+\cos\l(\f{\pi}{2n+2}\r)\r)}.
	\end{align*}
	
	To find an explicit expression for the minimizing kernel $u_n(k)$, recall that Chebyshev polynomials are orthogonal and in particular
	\begin{align*}
		\int_{-1}^{1} T_i(x) T_j(x) \f{dx}{\sqrt{1-x^2}} = 
		\begin{cases}
			0 \quad \text{if} \quad i\neq j,\\
			\pi \quad \text{if} \quad i = j = 0,\\
			\f{\pi}{2} \quad \text{if} \quad i = j \neq 0.\\
		\end{cases}
	\end{align*}
	Therefore for any $j$ we have
	\begin{align*}
		\int_{-1}^1 S_n(x)T_j(x) \f{dx}{\sqrt{1-x^2}} 
		= \int_{-1}^1\l(u_n(0) + \sum_{k=1}^n 2u_n(k)T_k(x)\r) \f{T_j(x) dx}{\sqrt{1-x^2}}
		= \pi \cdot u_n(j).
	\end{align*}
	Thus we can write $u_n(k)$ as
	\begin{align*}
		u_n(k) = \f{1}{\pi} \int_{-1}^1 S_n(x) T_k(x) \f{dx}{\sqrt{1-x^2}}
	\end{align*}
	where $T_k(x)$ is the $k$th Chebyshev polynomial, and
	\begin{align*}
		S_n(x) = \f{1}{x-1} \f{2\sin\l(\f{\pi}{2n+2}\r)}{(n+1)\l(1+\cos\l(\f{\pi}{2n+2}\r)\r)} T_{n+1}\l(\f{1+\cos\l(\f{\pi}{2n+2}\r)}{2}(x+1)-1\r).
	\end{align*}
\end{proof}

\subsection{The Epanechnikov Kernel}
This section is dedicated to proving Theorem~\ref{thm:epanechnikov}, which offers a kernel that is nearly optimal and easy to implement in practice. First we verify the Epanechnikov kernel $E_n: \{-n, \cdots, n\} \to \R$ satisfies the normalization requirement. Indeed, an induction argument gives the relation
\begin{align*}
	\sum_{k=-n}^n k^2 = \f{1}{3}n(n+1)(2n+1).
\end{align*} 
Therefore we can compute
\begin{align*}
	\sum_{k=-n}^{n} E_{n}(k)
	&= \sum_{k=-n}^{n}\f{3}{n(4 n^2 - 1)}\l(n^2-k^2\r)\\
	&= \f{3}{n(4n^2-1)}\l(n^2(2n+1)-\f{1}{3}n(n+1)(2n+1)\r) = 1.
\end{align*}

Next, note Lemma~\ref{prop:avg-to-poly} allows us to reduce the asymptotic relation of Theorem~\ref{thm:epanechnikov} to a claim about the polynomials
\begin{align*}
	p_n(x) = E_n(0) + 2\sum_{k=1}^{n} E_n(k)T_k(x).
\end{align*}
In particular, Lemma~\ref{prop:avg-to-poly} promises the equivalence
\begin{align*}
	\sup_{0 \neq f \in \ell^2(\Z)} \f{\ltwo{\Delta(E_n \ast f)}}{\ltwo{f}}
	= 2 \max_{-1 \leq x \leq 1}|(1-x)p_n(x)|.
\end{align*}
Therefore it suffices to show that as $n \to \infty$ we have the asymptotic equivalence
\begin{align*}
	\max_{-1 \leq x \leq 1}(1-x)p_n(x) \sim \f{3\mu}{2n^2}.
\end{align*}
For simplicity, we remove the normalizing coefficients of $p_n(x)$ and instead consider
\begin{align*}
	\til{p}_n(x) = \f{n(4n^2-1)}{3}p_n(x) = n^2+2\sum_{k=1}^{n-1}\l(n^2-k^2\r)T_k(x).
\end{align*}
Therefore it suffices to show the following asymptotic equivalence as $n \to \infty$:
\begin{align}
	\max_{-1 \leq x \leq 1}(1-x)\til{p}_n(x) \sim 2\mu n.
	\label{eq:polynomial-goal}
\end{align}
The first step in proving this asymptotic relation is to rewrite $(1-x)\til{p}(x)$ as follows.
\begin{lemma}
Define the polynomial
\[\til{p}_{n}(x) = n^2+2\sum_{k=1}^{n-1}\l(n^2-k^2\r)T_k(x).\]
Then,
\begin{equation}
(1-x)\til{p}_{n}(x) = \f{T_{n}(x)-T_{n-1}(x)}{x-1} + (1-2n)T_{n}(x).
\label{eq:equivalent-polynomial-form}
\end{equation}
\label{lemma:cheby-equivalence}
\end{lemma}
\begin{proof}
We first show the equality
\begin{equation}
	(1-x)\l(1+2\sum_{k=1}^{n-1}T_k(x)\r) = T_{n-1}(x)-T_n(x).
	\label{eq:sum-to-frac}
\end{equation}
This follows quickly by induction and the recurrence $T_{n+1}(x) = 2xT_n(x)-T_{n-1}(x)$. Indeed, for $n=1$ we find $(1-x)(1+2\cdot 0) = 1-x = T_0(x)-T_1(x).$
Furthermore, if (\ref{eq:sum-to-frac}) holds for $n$, then we find
\begin{align*}
	(1-x)\l(1+2\sum_{k=1}^n T_k(x)\r)
	&= (T_{n-1}(x) - T_n(x)) + 2(1-x)T_n(x)\\
	&= T_n(x)+(T_{n-1}(x)-2xT_n(x))\\
	&= T_n(x) - T_{n+1}(x).
\end{align*}
Due to (\ref{eq:sum-to-frac}), the claim follows so long as we can prove
\begin{equation}
(1-x)\til{p}_{n}(x) = \l(1 + 2\sum_{k=1}^{n-1}T_k(x)\r) + (1-2n)T_{n}(x).
\label{eq:1stequality}
\end{equation}
We prove (\ref{eq:1stequality}) by induction. Indeed, for $n=1$ we can immediately compute the  necessary relation
$(1-x)\til{p}_1 = 1-x = 1 + (1-2\cdot 1)T_1(x).$
Now suppose (\ref{eq:1stequality}) holds for $n$. First derive the following recurrence for $\til{p}_{n}$.
\begin{align*}
	\til{p}_{n+1}(x) 
	&= (n+1)^2+2\sum_{k=1}^n((n+1)^2-k^2)T_k(x)\\
	&= \l(n^2 + 2\sum_{k=1}^{n-1}(n^2-k^2)T_k(x)\r) + \l((2n+1) + 2\sum_{k=1}^{n}((n+1)^2-n^2)T_k(x)\r)\\
	&= \til{p}_{n}(x) + (2n+1)\l(1+2\sum_{k=1}^nT_k(x)\r).
\end{align*}
Therefore we find (\ref{eq:1stequality}) also holds for $n+1$ by the following computation, which uses our hypothesis and (\ref{eq:sum-to-frac}).
\begin{align*}
	(1-x)\til{p}_{n+1}(x) 
	&= (1-x)\l(\til{p}_{n}(x)+(2n+1)\l(1+\sum_{k=1}^nT_k(x)\r)\r)\\
	&= \l(1+2\sum_{k=1}^{n-1}T_k(x)+(1-2n)T_n(x)\r) + (2n+1)(T_{n}(x)-T_{n+1}(x))\\
	&= 1+2\sum_{k=1}^nT_k(x)+(1-2(n+1))T_{n+1}(x)
\end{align*}
\end{proof}
With this new form for $(1-x)\til{p}_{n}(x)$ in hand, we turn back to our objective of showing (\ref{eq:polynomial-goal}). For each $n$, we will bound the intervals $[-1,\cos\l(\f{16}{n}\r)]$ and $[\cos\l(\f{16}{n}\r),1]$ separately. Notice (\ref{eq:equivalent-polynomial-form}) is useful for bounding $[-1,\cos\l(\f{16}{n}\r)]$ because the polynomial $(T_n(x)-T_{n-1}(x))/(x-1)$ stays fairly small away from $1$, which the following claim quantifies.
\begin{lemma}
Over $[-1,1)$ we have
	\begin{align*}
		\l|\f{T_{n}(x) - T_{n-1}(x)}{x-1}\r| \leq \f{\sqrt{2}}{\sqrt{1-x}}.
	\end{align*}
	\label{lemma:beat-frequency-inequality}
\end{lemma}
\begin{proof}
	Write $T_{n}(x)-T_{n-1}(x) = \cos(n\theta)-\cos((n-1)\theta)$ by making the substitution $x = \cos\theta$. Then apply trigonometric identities to get the bound
	\begin{align*}
		|\cos(n\theta)-\cos((n-1)\theta)|
		&= \l|2\sin\l(\f{2n-1}{2}\theta\r)\sin\l(\f{\theta}{2}\r)\r|\\
		&\leq 2\l|\sin\l(\f{\theta}{2}\r)\r| = \sqrt{2} \cdot \sqrt{1-\cos\theta}.
	\end{align*}
	Therefore, substituting back to $x = \cos(\theta)$ we find
	\begin{align*}
		\l|\f{T_{n}(x) - T_{n-1}(x)}{x-1}\r|
		\leq \f{\sqrt{2}\sqrt{1-x}}{1-x} = \f{\sqrt{2}}{\sqrt{1-x}}.
	\end{align*}
\end{proof}
Now let $x \in [-1, \cos\l(\f{16}{n}\r)]$ and use (\ref{eq:equivalent-polynomial-form}) to bound
\begin{align}
	|(1-x)\til{p}_n(x)|
	&\leq (2n-1)|T_n(x)| + \l|\f{T_n(x)-T_{n-1}(x)}{x-1}\r| \label{eq:pn-ineq-chain}
\\ 
	&\leq (2n-1) + \f{\sqrt{2}}{\sqrt{1-x}}
	\leq (2n-1) + \f{\sqrt{2}}{\sqrt{1-\cos\l(\f{16}{n}\r)}}.\nonumber
\end{align}

We can simplify this further by taking the Taylor expansion of cosine about $0$ and computing the following limit.
\begin{align*}
	\lim_{n \to \infty} (2n-1)\f{\sqrt{1-\cos\l(\f{16}{n}\r)}}{\sqrt{2}}
	= \lim_{n \to \infty} \sqrt{\f{(2n-1)^2}{2}\l(\f{1}{2}\l(\f{16}{n}\r)^2 + O\l(\f{1}{n^{4}}\r)\r)}
	= 16.
\end{align*}
Therefore, returning to the inequality chain (\ref{eq:pn-ineq-chain}), we have for $x \in [-1, \cos\l(\f{16}{n}\r)]$ the following asymptotic equivalence as $n \to \infty$:
\begin{align*}
	|(1-x)\til{p}_n(x)| \sim (2n-1)\l(1+\f{1}{16}\r) \leq 2\mu n.
\end{align*}

Where the last inequality follows from $1+\f{1}{16} \leq \mu \approx 1.063$ by (\ref{eq:mu-def}). We have shown
$$\limsup_{n \to \infty} \max_{-1 \leq x \leq \cos\l(\f{16}{n}\r)} \f{1}{2n} (1-x)\til{p}_n(x) \leq \mu.$$

Therefore if only we can show the asymptotic equivalence
\begin{equation}
	\max_{\cos\l(\f{16}{n}\r) \leq x \leq 1} (1-x) \til{p}_n(x) \sim 2\mu n
	\label{eq:2nd-interval}
\end{equation}
as $n \to \infty$, then (\ref{eq:polynomial-goal}) and hence Theorem~\ref{thm:epanechnikov} follows. Substitute $x = \cos\l(\f{\alpha}{n}\r)$ and we claim uniform convergence
$$\f{1}{2n} \l(1-\cos\l(\f{\alpha}{n}\r)\r) \cdot \til{p}_n\l(\cos \l(\f{\alpha}{n}\r)\r) \to \f{\sin(\alpha)}{\alpha} - \cos(\alpha).$$ 
If this is true, the result follows immediately because uniform convergence allows the following interchange of maximum and limit.
\begin{align*}
	\lim_{n \to \infty} \f{1}{2n} \max_{\cos\l(\f{16}{n}\r) \leq x \leq 1}  |(1-x)\til{p}_n(x)|
	&= \lim_{n \to \infty}\max_{\alpha \in [0, 16]} \f{1}{2n} \l|\l(1-\cos\l(\f{\alpha}{n}\r)\r)\til{p}_n\l(\cos\l(\f{\alpha}{n}\r)\r)\r|\\
	&= \max_{\alpha \in [0,16]}\lim_{n \to \infty}\f{1}{2n} \l|\l(1-\cos\l(\f{\alpha}{n}\r)\r)\til{p}_n\l(\cos\l(\f{\alpha}{n}\r)\r)\r|\\
	&= \max_{\alpha \in [0,16]}\l|\f{\sin(\alpha)}{\alpha}-\cos(\alpha)\r| = \mu.
\end{align*}
Thus the following claim is the only remaining barrier to the proof of Theorem~\ref{thm:epanechnikov}.

\begin{lemma}
	We have uniform convergence
	$$\f{1}{2n} \l(1-\cos\l(\f{\alpha}{n}\r)\r) \til{p}_n\l(\cos\l(\f{\alpha}{n}\r)\r) \to \f{\sin(\alpha)}{\alpha} - \cos(\alpha)$$ as $n \to \infty$ over $\alpha \in [0,16]$.
	\label{lemma:uniform-convergence}
\end{lemma}
\begin{proof}
	Use (\ref{eq:equivalent-polynomial-form}) and trigonometric identities to rewrite the sequence of functions as
	\begin{align*}
		\f{1}{2n} \l(1-\cos\l(\f{\alpha}{n}\r)\r)\til{p}_n\l(\cos\l(\f{\alpha}{n}\r)\r)
		&= \f{1}{2n} \f{\cos(\alpha)-\cos\l(\f{n-1}{n}\alpha\r)}{\cos(\f{\alpha}{n})-1} + \f{1-2n}{2n}\cos(\alpha)\\
		&= \f{1}{n} \f{\sin\l(\f{2n-1}{2n}\alpha\r)\sin\l(\f{\alpha}{2n}\r)}{1-\cos\l(\f{\alpha}{n}\r)} - \f{2n-1}{2n}\cos(\alpha)\\
		&= \f{1}{n}\f{\sin\l(\f{2n-1}{2n}\alpha\r)}{2\sin\l(\f{\alpha}{2n}\r)}- \f{2n-1}{2n}\cos(\alpha)\\
	\end{align*}
	Because $\f{2n-1}{2n} \cos(\alpha) \to \cos(\alpha)$ uniformly, we only must show that
	\begin{align*}
		\f{1}{n}\f{\sin\l(\f{2n-1}{2n}\alpha\r)}{2\sin\l(\f{\alpha}{2n}\r)} \to \f{\sin(\alpha)}{\alpha}
	\end{align*}
	uniformly. We break this into two parts, first using that sine is Lipschitz with constant $1$ to compute
	\begin{align*}
		\l|\f{\sin\l(\f{2n-1}{2n}\alpha\r)}{\alpha} - \f{\sin(\alpha)}{\alpha}\r|
		\leq \l|\f{\f{2n-1}{2n}\alpha - \alpha}{\alpha}\r|
		= \l|\f{2n-1}{2n}-1\r|.
	\end{align*}
	 Notice the right hand side of the above inequality converges uniformly to $0$ as $n \to \infty$, giving the uniform convergence $\f{1}{\alpha} \sin(\f{2n-1}{2n}\alpha) \to \f{1}{\alpha}\sin (\alpha)$.
	Secondly, use Taylor's theorem to get the following uniform convergence over $\alpha \in [0,16]$:
	\begin{align*}
		\f{\alpha}{2n\sin\l(\f{\alpha}{2n}\r)}
		= \f{\alpha}{\alpha+O((\f{\alpha}{n})^3)}
		= \f{1}{1+O(\f{\alpha^2}{n^3})} \to 1.
	\end{align*}
	Because $1$ and $\f{1}{\alpha}\sin(\alpha)$ are bounded over $[0,16]$, the product of the sequences converges uniformly to the product of the limits and so we get uniform convergence
	\begin{align*}
	\f{1}{n}\f{\sin\l(\f{2n-1}{2n}\alpha\r)}{2\sin\l(\f{\alpha}{2n}\r)}
	= \f{\sin\l(\f{2n-1}{2n}\alpha\r)}{\alpha} \cdot \f{\alpha}{2n\sin\l(\f{\alpha}{2n}\r)} \to \f{\sin \alpha}{\alpha}.
	\end{align*}
\end{proof}	
	
\section*{Acknowledgements}
	This material is based upon work supported by the National Science Foundation Graduate Research Fellowship under Grant No. DGE-2140004. Additional thanks to Stefan Steinerberger for bringing this problem to the author's attention and for helpful conversations.

\bibliographystyle{plain}
\bibliography{bib}

\end{document}